\documentclass[reqno]{amsart}

\usepackage{amsmath,amssymb,amsthm,mathrsfs}

\theoremstyle{plain}
\numberwithin{equation}{section}
\theoremstyle{plain}
\newtheorem{theorem}{Theorem}[section]
\newtheorem{corollary}[theorem]{Corollary}
\newtheorem{lemma}[theorem]{Lemma}
\newtheorem{proposition}[theorem]{Proposition}
\theoremstyle{definition}

\newtheorem{example}[theorem]{Example}

\newcommand{\ands}{\quad\mbox{and}\quad}
\newcommand{\kernel}{{\mathrm{Ker}}}
\newcommand{\degr}{{\mathrm{deg} }}
\newcommand{\Dom}{{\mathrm{Dom}}}
\newcommand{\Rat}{{\mathrm{Rat}}}
\newcommand{\Ran}{{\mathrm{Ran}}}

\newcommand{\cP}{{\mathcal P}}
\newcommand{\cQ}{{\mathcal Q}}

\newcommand{\cH}{{\mathcal H}}
\newcommand{\cD}{{\mathcal D}}
\newcommand{\BC}{{\mathbb C}}
\newcommand{\BT}{{\mathbb T}}
\newcommand{\BR}{{\mathbb R}}
\newcommand{\BD}{{\mathbb D}}
\newcommand{\BN}{{\mathbb N}}
\newcommand{\BP}{{\mathbb P}}
\newcommand{\BZ}{{\mathbb Z}}

\newcommand{\mat}[2]{\ensuremath{\left[\begin{array}{#1}#2\end{array}\right]}}

\newcommand{\ov}[1]{{\overline{#1}}}

\newcommand{\inn}[2]{\ensuremath{\langle #1,#2 \rangle}}
\newcommand{\tu}[1]{\textup{#1}}

\newcommand{\wtil}[1]{{\widetilde{#1}}}

\newcommand{\al}{\alpha}

\newcommand{\ga}{\gamma}
\newcommand{\de}{\delta}

\newcommand{\la}{\lambda}

\newcommand{\si}{\sigma}

\newcommand{\vph}{\varphi}
\newcommand{\om}{\omega}

\begin{document}

\title[A Toeplitz-like operator with symbol having poles on the unit circle III]{A Toeplitz-like operator with rational symbol having poles on the unit circle III: the adjoint}

\author[G.J. Groenewald]{G.J. Groenewald}
\address{G.J. Groenewald, Department of Mathematics, Unit for BMI, North-West
University, Potchefstroom, 2531 South Africa}
\email{Gilbert.Groenewald@nwu.ac.za}

\author[S. ter Horst]{S. ter Horst}
\address{S. ter Horst, Department of Mathematics, Unit for BMI, North-West
University, Potchefstroom, 2531 South Africa}
\email{Sanne.TerHorst@nwu.ac.za}

\author[J. Jaftha]{J. Jaftha}
\address{J. Jaftha, Numeracy Centre, University of Cape Town, Rondebosch 7701; Cape Town; South Africa}
\email{Jacob.Jaftha@uct.ac.za}

\author[A.C.M. Ran]{A.C.M. Ran}
\address{A.C.M. Ran, Department of Mathematics, Faculty of Science, VU university Amsterdam, De Boelelaan 1081a, 1081 HV Amsterdam, The Netherlands and Unit for BMI, North-West~University, Potchefstroom, South Africa}
\email{a.c.m.ran@vu.nl}

\thanks{This work is based on the research supported in part by the National Research Foundation of South Africa (Grant Number 90670 and 93406).}



\begin{abstract}
This paper contains a further analysis of the Toeplitz-like operators $T_\om$ on $H^p$ with rational symbol $\om$ having poles on the unit circle that were previously studied in \cite{GtHJR1,GtHJR2}. Here the adjoint operator $T_\om^*$ is described. In the case where $p=2$ and $\om$ has poles only on the unit circle $\BT$, a description is given for when $T_\om^*$ is symmetric and when $T_\om^*$ admits a selfadjoint extension. Also in the case where $p=2$, $\om$ has only poles on $\BT$ and in addition $\om$ is proper, it is shown that $T_\om^*$ coincides with the unbounded Toeplitz operator defined by Sarason in \cite{S08}.
\end{abstract}

\maketitle

\section{Introduction}

In this paper we proceed with our study of unbounded Toeplitz-like operators on $H^p$ with rational symbols that have poles on the unit circle $\BT$ which was initiated in \cite{GtHJR1}. Our previous work on such Toeplitz-like operators focused on their Fredholm properties (in \cite{GtHJR1}) and the various parts of their spectra (in \cite{GtHJR2}). Here we determine properties of the adjoint operator and conditions under which the operator is symmetric and when it has a selfadjoint extension.

Before we can define our Toeplitz-like operators, some notation has to be introduced. We write $\Rat$ for the space of rational complex functions, $\Rat(\BT)$ for the subspace of $\Rat$ consisting of rational complex functions with poles only on the unit circle $\BT$, and $\Rat_0(\BT)$ for the subspace of strictly proper functions in $\Rat(\BT)$. Now let $\om\in\Rat$, possibly with poles on $\BT$. As in \cite{GtHJR1}, we define the Toeplitz-like operator $T_\om\, (H^p\to H^p)$, for $1<p<\infty$, via
\begin{equation}\label{Tom}
\Dom(T_\om)=\{ g\in H^p\mid \om g =f+\rho,\mbox{ with $f\in L^p,\, \rho\in \Rat_0(\BT)$}\},\quad T_\om g=\BP f.
\end{equation}
Here $\BP$ is the Riesz projection of $L^p$ onto $H^p$. The operator $T_\om$ is densely defined and closed. In case $\om\in \Rat(\BT)$, explicit formulas for the domain, kernel, range, and a complement of the range were obtained in \cite{GtHJR2}, as an extension of a result in \cite{GtHJR1} for the case where $T_\om$ is Fredholm. We briefly recall these results in Section \ref{S:Tom*}, as they will be frequently used throughout the paper.

In case $\om$ has no poles on $\BT$, in fact for any $\om\in L^\infty$, the adjoint of the Toeplitz operator $T_\om$ on $H^p$ can be identified with the Toeplitz operator $T_{\om^*}$ on $H^{p'}$, with $1<p'<\infty$ such that $1/p + 1/p' =1$ and with $\om^*$ defined as $\om^*(z)=\ov{\om(z)}$ on $\BT$. The identification of $(H^p)'$ and $H^{p'}$ goes via the usual pairing
\[
\inn{f}{g}_{p,p'}=\frac{1}{2\pi}\int_{\BT} \ov{g(z)} f(z)\, dz\quad (f\in H^p, g\in H^{p'}).
\]
In the sequel we use the same notation for the similarly defined pairing between $L^p$ and $L^{p'}$ to identify $(L^p)'$ and $L^{p'}$, and in both cases the indices will often be omitted.

For the Toeplitz-like operators studied in this paper the situation is more complicated than for Toeplitz operators with $L^\infty$ symbols. However, we do obtain that $T_{\om}^*$ can be identified with the restriction of the Toeplitz-like operator $T_{\om^*}$ on $H^{p'}$ to a dense subspace of its domain. Like for the operator $T_\om$, in case $\om$ is in $\Rat(\BT)$ we obtain a more explicit description of $T_{\om}^*$, which we present after introducing some further notation.

Throughout the paper $\cP$ denotes the space of complex polynomials and $\cP_k$, for any non-negative integer $k$, denotes the subspace of $\cP$ of polynomials of degree at most $k$. The degree of a polynomial $r\in\cP$ is denoted as $\degr(r)$. Given $r\in\cP$ with $\degr(r)=k$, say $r(z)=r_0+z r_1+\cdots + z^k r_k$, we define the polynomial $r^\sharp$ by
\[
r^\sharp(z)=z^k \ov{r(1/\ov{z})} =\overline{r_0}z^k + \ov{r_1} z^{k-1}+\cdots + \ov{r_k}.
\]

The following theorem is our first main result.

\begin{theorem}\label{T:main1}
Let $\om = s/q\in\Rat$ with $s,q\in\cP$ co-prime and $1<p<\infty$. Factor $s=s_-s_0s_+$ and $q=q_-q_0q_+$ with $s_-,q_-$ having roots only inside $\BT$, $s_0,q_0$ having roots only on $\BT$, and $s_+,q_+$ having roots only outside $\BT$. Set $m=\deg(q)$, $n=\deg(s)$,  $m_{\pm}=\degr(q_{\pm})$, $n_{\pm}=\degr(s_{\pm})$ $m_{0}=\degr(q_{0})$, $n_{0}=\degr(s_{0})$ and let $1<p'<\infty$ with $1/p + 1/p'=1$. Then
\begin{equation}\label{Tom*}
\Dom (T_\om^*) = (q_0)^\sharp H^{p'}\subset \Dom(T_{\om^*}) \ands
T_\om^* = T_{\om^*}|_{(q_0)^\sharp H^{p'}}.
\end{equation}
Furthermore, we have
\begin{equation}\label{Tom*RanKer}
\begin{aligned}
\Ran (T_\om^*) &= T_{z^{m-n}(s_+)^\sharp/(q_+)^\sharp} Q_{n_0+n_--m_0-m_-} (s_0)^\sharp H^{p'},\\
\kernel (T_\om^*) &= \left\{ \frac{(q_-)^\sharp (q_0)^\sharp r}{(s_-)^\sharp} \mid \deg(r) < n_- - m_- - m_0 \right \}.
\end{aligned}
\end{equation}
Here $Q_{k}=I_{H^{p'}}-P_{\cP_{k-1}}$, with $P_{\cP_{k-1}}$ the standard projection in $H^{p'}$ onto $\cP_{k-1}\subset H^{p'}$ to be interpreted as $0$ if $k\leq 0$, i.e., $Q_{k}=I_{H^{p'}}$ if $k\leq 0$. Thus, for $n_0+n_- \leq m_0+m_-$ we have $\Ran (T_\om^*) = T_{z^{m-n}/(q_+)^\sharp} (s_+s_0)^\sharp H^{p'}$.
Moreover,
\[
\dim \kernel (T_\om^*) = \max\left \{0, \#\{\textrm{zeroes of } \om \textrm{ inside } \BD \} - \#\{\textrm{poles of } \om \textrm{ in } \overline{\BD}\} \right \},
\]
where the multiplicities of the zeroes and poles are taken into account. Hence, $\dim \kernel(T_\om^*)$ is the maximum of $0$ and $n_- -m_- - m_0$. In particular, $T_{\om}^*$ is injective if and only if the number of poles of $\om$ inside $\overline{\BD}$ is greater than or equal to the number of zeroes of $\om$ inside $\BD$, multiplicities taken into account.
\end{theorem}

Before giving a proof of Theorem \ref{T:main1} in Section \ref{S:general}, we prove the specialization of this result for the case $\om\in\Rat(\BT)$ in Section \ref{S:adjoint1}. For this purpose we first provide a description of $T_{\om^*}$ in Section \ref{S:Tom*}.

The injectivity result, but not the description of $\kernel(T_\om^*)$, can also be derived from general theory and results on $T_\om$. Indeed, according to Theorem II.3.7 in \cite{Goldberg}, $T_\om^*$ is injective if and only if $T_\om$ has dense range, so that the claim follows from Proposition 2.4 in \cite{GtHJR2}. More can be obtained in this way, since $H^p$, $1<p<\infty$, is reflexive. By Theorem II.2.14 of \cite{Goldberg} it follows that $T_\om^{**}=T_\om$, with the usual identifications of the dual spaces. Hence, applying the above to $T_\om^*$ we find that $T_\om^*$ has dense range if and only if $T_\om$ is injective; see also Theorem II.4.10 in \cite{Goldberg}. By Banach's Closed Range Theorem, cf., \cite{Y80}, $T_\om^*$ has closed range if and only if $T_\om$ has closed range. Again applying results from \cite{GtHJR2} now gives the following result.

\begin{corollary}\label{C:close/dense}
Let $\om\in\Rat$ and $1<p<\infty$. Then $T_\om^*$ has closed range if and only if $\om$ has no zeroes on $\BT$, or equivalently, $\om^*$ has no zeroes on $\BT$. Moreover, $T_\om^*$ has dense range if and only if
\[
\# \left\{\begin{array}{l}\!\!\!
 \textrm{poles of } \om \textrm{ inside }\overline{\BD} \!\!\! \\
\!\!\!\textrm{multi.\ taken into account}\!\!\!
\end{array}\right\}\leq
\# \left\{\begin{array}{l}\!\!\!
 \textrm{zeroes of } \om \textrm{ inside }\overline{\BD} \!\!\! \\
\!\!\!\textrm{multi.\ taken into account}\!\!\!
\end{array}\right\}.
\]
\end{corollary}

Beyond Section \ref{S:general}, and in the remainder of this introduction, we only consider the case $p=2$ and $\om\in\Rat(\BT)$. By comparing the results on $T_\om$ and $T_\om^*$ it is obvious $T_\om$ cannot be selfadjoint, except when $\om$ has no poles on $\BT$. In Section \ref{S:symmetric} we describe in terms of $\om$ when $T_\om^*$ is symmetric, in which case $T_\om^*\subset T_\om$, and whenever $T_\om^*$ is symmetric we describe when $T_{\om^*}$ admits a selfadjoint extension. The following theorem collects some of the main results of Section \ref{S:symmetric}; it follows directly from Theorem \ref{T:symm}, Corollaries \ref{C:degcons} and \ref{C:notR}, Propositions \ref{P:selfadjext} and  \ref{P:selfadjext3}.

\begin{theorem}\label{T:main2}
Let $\om=s/q\in\Rat(\BT)$ with $s,q\in\cP$ co-prime. Consider $T_\om$ on $H^2$. Then
\[
\mbox{$T_\om^*$ is symmetric}\quad \Longleftrightarrow \quad \om(\BT)\subset \BR.
\]
In particular, if $T_\om^*$ is symmetric, then $\deg(s)\leq \deg(q)\leq 2\deg(s)$. Furthermore, if $T_\om^*$ is symmetric, then $T_\om^*$ admits a selfadjoint extension if and only if the number of roots of $s-iq$ and $s+iq$ in $\BD$, counting multiplicities, coincide. This happens in particular if $\om(\BT)\neq \BR$, but cannot happen in case $\deg(q)$ is odd.
\end{theorem}

Several other conditions for $T_\om^*$ to be symmetric and/or have a selfadjoint extension are derived in Section \ref{S:symmetric}.

In \cite{S08} Sarason introduced and studied an unbounded Toeplitz-like operator with symbol in the Smirnov class. In Section \ref{S:Sarason} we show that if $\om\in\Rat(\BT)$ is proper, then the adjoint operator $T_\om^*$ is precisely a Toeplitz-like operator of the type studied by Sarason. Hence in this case our Toeplitz-like operator $T_\om=T_\om^{**}$ coincides with the adjoint of the Toeplitz-like operator considered in \cite{S08}. Based on ideas in \cite{S08}, we also show that $H(\overline{\BD})$, the space of functions analytic on a neighborhood of $\overline{\BD}$, is contained in $\Dom(T_\om)$ and in fact is a core of $T_\om$.

In the last section of \cite{S08}, Sarason introduces a class of closed, densely defined Toeplitz-like operators on $H^2$ determined by algebraic properties, which was further investigated by Rosenfeld in \cite{R13,R16}. In particular, this class of Toeplitz-like operators contains the unbounded Toeplitz-like operator studied by Sarason and is closed under taking adjoints, and hence contains our Toeplitz-like operators with proper symbols in $\Rat(\BT)$. In fact, we will show in Section \ref{S:Sarason} that $T_\om$ is contained in the class of Toeplitz-like operators for any $\om$ in $\Rat$.

\section{The operator $T_{\om^*}$ for $\om\in \Rat(\BT)$}\label{S:Tom*}

In this section we recall some results from \cite{GtHJR1,GtHJR2} on the operator $T_\om$ for $\om\in \Rat(\BT)$ that we will use in the sequel, and apply them to the operator $T_{\om^*}$. Hence, throughout this section let $\om=s/q\in\Rat(\BT)$, with $s,q\in\cP$ co-prime. We set $m=\deg(q)$ and $n=\degr(s)$. Furthermore, factor $s=s_-s_0s_+$ with $s_-$, $s_0$ and $s_+$ polynomials having roots only inside, on, or outside $\BT$, respectively. We then recall from Theorem 2.2 in \cite{GtHJR2} that
\begin{equation}\label{DomRanId}
\begin{aligned}
&\qquad \kernel (T_\omega) = \left\{r/s_+ \mid \deg(r) < m - \deg(s_-s_0) \right\};\\
&\Dom(T_\om)=qH^p+\cP_{m-1}; \quad
\Ran(T_\om)=s H^p+\wtil\cP,
\end{aligned}
\end{equation}
where $\wtil\cP$ is the subspace of $\cP$ given by
\begin{equation}\label{tilP}
\wtil\cP = \{ r\in\cP \mid r q = r_1 s + r_2 \mbox{ for } r_1,r_2\in\mathcal{P}_{m-1}\}\subset \cP_{n-1}.
\end{equation}
Furthermore, $H^p=\overline{\Ran(T_\om)} + \wtil{\cQ}$ forms a direct sum decomposition of $H^p$, where
\begin{equation}\label{tilQ}
\wtil\cQ=\cP_{k-1}\quad \mbox{with}\quad k=\max\{\deg(s_-)-m , 0\},
\end{equation}
following the convention $\cP_{-1}:=\{0\}$. Furthermore, the action of $T_\om$ is as follows.
\begin{align*}
& T_{\om}g =sh +\wtil{r} \quad (\mbox{$g= qh+r\in qH^p+\cP_{m-1}=\Dom(T_\om)$}),\\
& \mbox{where $\wtil{r}\in\cP_{n-1}$ is such that $rs=\wtil{r}q+r_2$ for some $r_2\in\cP_{m-1}$.}
\end{align*}

We also recall from Lemma 5.3 in \cite{GtHJR1} that
\begin{equation}\label{zkappaTom}
T_{z^\kappa \om}=T_{z^\kappa} T_{\om}\ \  \mbox{for any integer $\kappa\leq 0$.}
\end{equation}

Recall that $\om^*$ is defined as $\om^*(z)=\ov{\om(z)}$ on $\BT$, i.e.,
$\om^*(z)=\overline{s(z)}\slash\overline{q(z)}$.
For $z\in\BT$
\[
\overline{q(z)}= \overline{q_0+zq_1+ \cdots +z^mq_m} =
\overline{q_0}+\overline{q_1}\frac{1}{z}+\cdots +\overline{q_m}\frac{1}{z^m}=
\frac{1}{z^m}q^\sharp(z).
\]
Hence $q^\sharp(z)=z^m \ov{q(z)}$, and likewise $s^\sharp(z)= z^n \ov{s(z)}$. Thus we have
\begin{equation}\label{om*}
\om^*(z)=\frac{z^{m-n}s^\sharp(z)}{q^\sharp(z)} \mbox{ if $m\geq n$} \ands
\om^*(z)=\frac{s^\sharp(z)}{z^{n-m}q^\sharp(z)} \mbox{ if $m< n$}.
\end{equation}
In fact, the formula $\om^*(z)=z^{m-n}s^\sharp(z)/q^\sharp(z)$ holds in both cases, but is not always a representation as the ratio of two polynomials. Note in particular that $\om^*\in \Rat(\BT)$ in case $\om$ is proper, while this need not be the case if $\om$ is not proper. Thus, if $\om$ is proper, the above formulas apply directly, while for the non-proper case, using \eqref{zkappaTom} we can reduce certain questions to questions concerning the Toeplitz operator $T_{s^\sharp/ q^\sharp}$ with symbol $s^\sharp/ q^\sharp$ which is in $\Rat(\BT)$.

A polynomial $r\neq 0$ is called self-inversive in case $r=\ga r^\sharp$ for a constant $\ga\in\BC$, which necessarily is unimodular. In fact, $\ga$ is the ratio $r_0/\overline{r_n}$ with $r_0=r(0)$ and $r_n$ the leading coefficient of $r$.  By a theorem of Cohn \cite{C22}, a polynomial $r$ has all its roots on $\BT$ if and only if $r$ is self-inversive and its derivative has all its roots in the closed unit disc $\overline{\BD}$. Hence, any polynomial with roots only on $\BT$ is self-inversive. In particular,  $q= \ga q^\sharp$ and $s_0=\rho (s_0)^\sharp$ for unimodular constants $\ga$ and $\rho$.

More generally, in the transformation $r\to r^\sharp$, the nonzero roots of $r$ (including multiplicity) transfer along the unit circle via the map $\al \mapsto 1/\ov{\al}=|\al|^{-2} \al$, while the degree decreases by the multiplicity of 0 as a root of $r$. Consequently, in the factorization $s^\sharp= (s_+)^\sharp (s_0)^\sharp (s_-)^\sharp$, the polynomials $(s_+)^\sharp$, $(s_0)^\sharp$ and $(s_-)^\sharp$ contain the roots of $s^\sharp$ inside, on and outside $\BT$, respectively, taking multiplicities into account. We write $(s_+)^\sharp$ rather than $s_+^\sharp$, etc., to avoid confusion with what one may interpret as $(s^\sharp)_+$.

We now apply the above to $T_{\om^*}$ acting on $H^{p'}$, $1<p'<\infty$, to fit better with the remainder of the paper.

\begin{proposition}\label{P:Tom*}
Let $\om= s/q\in \Rat(\BT)$, with $s,q\in\cP$ co-prime, $m=\deg(q)$ and $n=\deg(s)$. Factor $s=s_-s_0s_+$ with $s_-$, $s_0$ and $s_+$ polynomials having roots only inside, on, or outside $\BT$, respectively. Then for $T_{\om^*}$ on $H^{p'}$, $1<p'<\infty$, we have
\begin{equation}\label{DomKerIds*}
\kernel (T_{\omega^*}) = \left\{r_0/(s_-)^\sharp \mid \deg(r_0) < \deg(s_-) \right\},\ \  \Dom(T_{\om^*})=q^\sharp H^{p'}+\cP_{m-1}.
\end{equation}
Moreover, we have
\begin{equation}\label{RanIds*}
\begin{aligned}
\Ran(T_{\om^*})&= z^{m-n} s^\sharp H^{p'}+\wtil\cP_*\quad\mbox{ if $m\geq n$},\\
\Ran(T_{\om^*})&=T_{z^{m-n}}( s^\sharp H^{p'}+\wtil\cP_*)\quad \mbox{ if $m< n$},
\end{aligned}
\end{equation}
where for $m\geq n$ the subspace $\wtil\cP_*$ is given by
\begin{equation}\label{tilP*1}
\wtil\cP_* = \{ r\in\cP \mid r q^\sharp = z^{m-n} r_1 s^\sharp + r_2 \mbox{ for } r_1,r_2\in\mathcal{P}_{m-1}\}\subset \cP_{m-n+\deg(s^\sharp)-1},
\end{equation}
while for $m< n$ we have
\begin{equation}\label{tilP*2}
\wtil\cP_* = \{ r\in\cP \mid r q^\sharp = r_1 s^\sharp + r_2 \mbox{ for } r_1,r_2\in\mathcal{P}_{m-1}\}\subset \cP_{\deg(s^\sharp)-1}.
\end{equation}
Furthermore, $\Ran(T_{\om^*})$ is dense in $H^{p'}$.
\end{proposition}

\begin{proof}[\bf Proof] We separate the cases $m\geq n$ and $m< n$.

For $m\geq n$, we have $\om^*=\wtil{s}/\wtil{q}\in\Rat(\BT)$ with $\wtil{s}=z^{m-n}s^\sharp$ and $\wtil{q}=q^\sharp$. Hence $\wtil{s}$ factors as $\wtil{s}=(z^{m-n}(s_+)^\sharp) (s_0)^\sharp (s_-)^\sharp$, where the factors have all their roots inside, on, or outside $\BT$, respectively. Also, $\deg(q^\sharp)=\deg(q)$ and $\deg((s_+)^\sharp)=\deg(s_+)$. So the formulas for $\Dom(T_{\om^*})$ and $\Ran(T_{\om^*})$ follow directly from \eqref{DomRanId}, while the formula for $\kernel(T_{\om^*})$ follows because the bound on the degree of $r_0$ can be computed as
\begin{align*}
m-\deg(z^{m-n}(s_+)^\sharp (s_0)^\sharp)    & = n- \deg((s_+)^\sharp (s_0)^\sharp)
=n-\deg(s_+s_0)=\deg(s_-).
\end{align*}
Finally, a complement of the closure of $\Ran(T_{\om^*})$ is given by $\cP_{k-1}$ with $k$ the maximum of $0$ and $\deg(z^{m-n}(s_+)^\sharp)-m=\deg ((s_+)^\sharp)-n\leq 0$. Hence $\cP_{-1}=\{0\}$. Thus $T_{\om^*}$ has dense range, as claimed.

In case $m< n$, we have $T_{\om^*}=T_{z^{m-n}}T_{s^\sharp/ q^\sharp}$ and $s^\sharp/q^\sharp$ is in $\Rat(\BT)$. Applying the above results for $T_\om$ to $T_{s^\sharp/ q^\sharp}$ directly gives the formulas for $\Dom(T_{\om^*})$ and $\Ran(T_{\om^*})$.

To see that the formula for $\kernel(T_{\om^*})$ holds, we follow the argumentation of the proof of Lemma 4.1 in \cite{GtHJR1}. For $g\in \Dom(T_{\om^*})=\Dom(T_{s^\sharp/q^\sharp})$ to be in $\kernel (T_{\om^*})$ is equivalent to $T_{s^\sharp/ q^\sharp} g \in \cP_{n-m-1}$.
In other words, by Lemma 3.2 in \cite{GtHJR1}, to $s^\sharp g= q^\sharp \wtil{r} +r_1$ with $r_1\in \cP_{m-1}$ and $\wtil{r}\in \cP_{n-m-1}$, since then $T_{s^\sharp/ q^\sharp}g= \wtil{r}$. The latter happens precisely when $g=r/(s_-)^\sharp$ with $r\in\cP_{\deg(s_-)-1}$. Indeed, in that case $\deg((s_+)^\sharp(s_0)^\sharp r)<n$ which in the equation  $(s_+)^\sharp(s_0)^\sharp r=s^\sharp g= q^\sharp \wtil{r} +r_1$ corresponds to $\deg(\wtil{r})< m-1$, as required. Finally, we note that a complement of $\ov{\Ran(T_{s^\sharp/q^\sharp})}$ in $H^{p'}$ is given by $\cP_{k-1}$ with $k=\max\{0, \deg{s_+}^\sharp-m\}\leq n-m$. Let $f\in H^{p'}$ and write $z^{n-m}f= h+r\in \ov{\Ran(T_{s^\sharp/q^\sharp})}+ \cP_{k-1}$. Then $f=T_{z^{m-n}} z^{n-m}f=T_{z^{m-n}} (h+r)=T_{z^{m-n}} h\in T_{z^{m-n}} \ov{\Ran(T_{s^\sharp/q^\sharp})}\subset \ov{\Ran( T_{z^{m-n}}T_{s^\sharp/q^\sharp})}=\ov{\Ran(T_{\om^*})}$.  Thus also in this case $\Ran(T_{\om^*})$ is dense in $H^{p'}$.
\end{proof}

We conclude this section with a lemma will be of use in the sequel.

\begin{lemma}\label{L:sumpound}
Let $r_1,r_2\in\cP$. Set $n_i=\deg(r_i)$, for $i=1,2$, and $n=\deg(r_1+r_2)$. Then
\[
(r_1+r_2)^\sharp =z^{n-n_1} r_1^\sharp + z^{n-n_2} r_2^\sharp.
\]
In case $n<\max\{n_1,n_2\}$, then $n_1=n_2$ and $0$ is a root of $r_1^\sharp+r_2^\sharp$ with multiplicity $n-n_1$, so that the left hand side in the above identity still is a polynomial without a root at $0$.
\end{lemma}

\begin{proof}[\bf Proof]
By definition, for $z\in\BT$ we have
\begin{align*}
  (r_1+r_2)^\sharp (z) &
  = z^{n} (\overline{r_1(1/\overline{z})} + \overline{r_2(1/\overline{z})})
=z^{n-n_1} z^{n_1}\overline{r_1(1/\overline{z})} + z^{n-n_2} z^{n_2} \overline{r_2(1/\overline{z})}\\
 & =z^{n-n_1} r_1^\sharp(z) + z^{n-n_2} r_2^\sharp(z).\qedhere
\end{align*}
\end{proof}

\section{The adjoint of $T_\om$ for $\om\in\Rat(\BT)$}\label{S:adjoint1}

In this section we prove the first main result, Theorem \ref{T:main1}, for the special case that $\om\in\Rat(\BT)$. In this case, the result specializes to the following theorem, which we prove in this section.

\begin{theorem}\label{T:main1'}
Let $\om = s/q\in\Rat(\BT)$ with $s,q\in\cP$ co-prime and $1<p<\infty$. Set $m=\degr(q)$ and $n=\degr(s)$ and let $1<p'<\infty$ with $1/p + 1/p'=1$. Then
\begin{equation}\label{Tom*'}
\Dom (T_\om^*) = q^\sharp H^{p'}\subset \Dom(T_{\om^*}) \ands
T_\om^* = T_{\om^*}|_{q^\sharp H^{p'}}.
\end{equation}
In fact, for $g=q^\sharp v\in q^\sharp H^{p'}$ we have $T_\om^*g =T_{z^{m-n}} s^\sharp v$. Moreover, factorize $s = s_- s_0 s_+$ with $s_-$, $s_0$ and $s_+$ polynomials having roots only inside, on, or outside $\BT$, respectively. Then
\begin{equation}\label{Tom*RanKer'}
\Ran (T_\om^*) = T_{z^{m-n}} s^\sharp H^{p'}
\mbox{ and }
\kernel (T_\om^*) = \left\{ \frac{q^\sharp r}{(s_-)^\sharp} \mid \deg(r) < \degr(s_-) - m \right \}.
\end{equation}
In particular, we have
\[
\dim \kernel (T_\om^*) = \max\left\{0, \#\left\{
\textrm{zeroes of } \om^* \textrm{ outside } \BT
\right\} -
\#\left\{
\textrm{poles of } \om^* \textrm{ on } \BT
\right\}
\right\},
\]
where the multiplicities of the zeroes and poles are taken into account.
Thus $T_{\om}^*$ is injective if and only if $\om$ has at least as many poles inside $\BT$ as zeroes inside $\BT$ unequal to $0$, multiplicities taken into account.
\end{theorem}

We first present some auxiliary lemmas. Throughout, let $1<p,p'<\infty$ such that $1/p+1/p^\prime=1$. We will consider $T_{\om}$ as an operator with domain in $H^p$ and $T_{\om^*}$ as an operator with domain in $H^{p'}$.

\begin{lemma}\label{L:incl1}
Let $\om = s/q\in\Rat(\BT)$ with $s,q\in\cP$ co-prime, $m=\degr(q)$ and $n=\degr(s)$. Then
\[
q^\sharp H^{p'} \subset \Dom(T_\om^*)\cap \Dom(T_{\om^*}) \ands
T_\om^*|_{q^\sharp H^{p'} }=T_{\om^*}|_{q^\sharp H^{p'} }.
\]
Moreover, for $g=q^\sharp v\in q^\sharp H^{p'}$, with $v\in H^{p'}$, we have $T_{\om}^* g= T_{z^{m-n}}s^\sharp v$, and thus $T_\om^* (q^\sharp H^{p'})=T_{z^{m-n}} s^\sharp H^{p'}$.
\end{lemma}

\begin{proof}[\bf Proof]
The inclusion $q^\sharp H^{p'} \subset \Dom(T_{\om^*})$ follows from Proposition \ref{P:Tom*}.
Let $g$ be in $q^\sharp H^{p'}$, say $g(z)=q^\sharp(z)v(z)$ for $v\in H^{p'}$. We show that for $f\in \Dom(T_\om)$ we have $\inn{T_w f}{g}_{p,p'}=\inn{f}{T_{\om^*} g}_{p,p'}$. Let $f\in\Dom(T_\om)$ and $h=T_{\om}f\in H^p$, i.e., $sf=qh+r$ for some $r\in \cP_{m-1}$, by \cite[Lemma 2.3]{GtHJR1}. Then
\begin{align*}
\langle T_\om f, g \rangle_{p,p'}
&=\langle h, q^\sharp v\rangle_{p,p'}=\langle h, z^m \overline{q} v\rangle_{p,p'} =\langle qh, z^m  v\rangle_{p,p'} \\
& = \langle sf-r, z^m v\rangle_{p,p'} =\langle sf, z^m v\rangle_{p,p'} \quad (\mbox{because  $\degr(r)<m$, $v\in H^{p'}$})\\
&=\!\langle f, z^m \overline{s}v\rangle_{p,p'} \!=\!\langle f, z^{m-n} s^\sharp v\rangle_{p,p'}
\!=\!\langle f, T_{z^{m-n}} s^\sharp v\rangle_{p,p'}\ (\textup{because } f\in H^p).
\end{align*}
It remains to show that $T_{\om^*}g=T_{z^{m-n}} s^\sharp v$. If $m\geq n$, then $\om^*=z^{m-n}s^\sharp/q^\sharp$ is in $\Rat(\BT)$ and $\om^* g=z^{m-n}s^\sharp v\in H^{p'}$, so that, $T_{\om^*}g=z^{m-n} s^\sharp v= T_{z^{m-n}} s^\sharp v$, by Lemma 2.3 in \cite{GtHJR1}. In case $m<n$, we have $T_{\om^*} g=T_{z^{m-n}}T_{s^\sharp/q^\sharp}g=T_{z^{m-n}} s^\sharp v$.
\end{proof}

\begin{lemma}\label{L:incl2}
Let $\om = s/q\in\Rat(\BT)$ with $s,q\in\cP$ co-prime, $m=\degr(q)$ and $n=\degr(s)$. Let $g\in\Dom(T_\om^*)$ and $k=T_\om^* g\in H^{p'}$. Then for any $r\in \cP_{n-1}$ and $r_1\in\cP_{m-1}$ so that
\begin{equation}\label{r1r}
sr_1=qr +r_2 \mbox{ for some } r_2 \in \cP_{m-1}
\end{equation}
we have
\[
\inn{r_1}{k}_{p,p'}=\inn{r}{g}_{p,p'}.
\]
Moreover, we have
\begin{equation}\label{polyincl}
z^{m-n}s^\sharp g- q^\sharp k\in\cP_{m-1} \mbox{ if $m\geq n$}\ands
s^\sharp g- z^{n-m}q^\sharp k\in\cP_{n-1}\mbox{ if $m< n$}.
\end{equation}
In particular, $ \Dom(T_\om^*) \subset \Dom(T_{\om^*})$ and  $T_\om^*=T_{{\om^*}}|_{\Dom(T_\om^*)}$.
\end{lemma}

\begin{proof}[\bf Proof]
Let $g\in\Dom(T_\om^*)$ and $k= T_\om^* g$. Hence $\inn{T_\om f}{g}_{p,p'}=\inn{f}{k}_{p,p'}$ for each $f\in\Dom(T_\om)$. Since $\om\in\Rat(\BT)$, we have $\Dom(T_\om)=q H^p+\cP_{m-1}$. Let $f=qh +r_1\in \Dom(T_\om)$, with $h\in H^p$ and $r_1\in \cP_{m-1}$. Then $T_\om f= sh +r$ where $r\in\cP_{n-1}$ is uniquely determined by \eqref{r1r}. Thus
\begin{align*}
\inn{sh}{g}+\inn{r}{g}=\inn{sh+r}{g}=\inn{T_\om f}{g}=\inn{f}{k}=\inn{qh+r_1}{k}=\inn{qh}{k}+\inn{r_1}{k}.
\end{align*}
We obtain that
\[
\inn{sh}{g}-\inn{qh}{k}=\inn{r_1}{k}-\inn{r}{g}.
\]
However, in choosing $f\in \Dom(T_\om)$ we can choose $h\in H^p$ and $r_1\in\cP_{m-1}$ independently, and in particular set one or the other equal to zero, resulting in
\begin{align*}
\inn{sh}{g}=\inn{qh}{k}\ \ \ (h\in H^p),\ \ \  &
\inn{r_1}{k}=\inn{r}{g}\ \ \ \mbox{($r\in\cP_{n-1},r_1\in\cP_{m-1}$ as in \eqref{r1r})}.
\end{align*}
The second identity proves the first claim of the lemma. From the first identity we obtain that
\begin{align*}
0 & = \inn{h}{\ov{s}g-\ov{q}k}_{p,p'}=\inn{h}{z^{-n}s^\sharp g-z^{-m}q^\sharp k}_{p,p'} \quad (h\in H^p).
\end{align*}
Thus $\BP(z^{-n}s^\sharp g-z^{-m}q^\sharp k)=0$. On the other hand, for  $l=\max\{m,n\}$ we have
\[
z^l(z^{-n}s^\sharp g-z^{-m}q^\sharp k)=z^{l-n}s^\sharp g-z^{l-m}q^\sharp k\in H^{p'}.
\]
This can only occur if $z^{l-n}s^\sharp g-z^{l-m}q^\sharp k\in \cP_{l-1}$, which proves the second claim.

To complete the proof, we show that $g\in\Dom(T_{\om^*})$ and $T_{\om^*}g=k$. For $m\geq n$ we have $\om^*\in \Rat(\BT)$ and the first inclusion of \eqref{polyincl} can be rewritten as
\[
\om^* g= \left(\frac{z^{m-n} s^\sharp}{q^\sharp}\right) g= k +\wtil{r}/q^\sharp,\quad \mbox{for some\ \  $\wtil{r}\in \cP_{m-1}$}.
\]
Since $\deg(q^\sharp)=\deg(q)=m$, it now follows that $g\in \Dom(T_{\om^*})$ and $T_{\om^*}g=k$. In case $m<n$ we have $T_{\om^*}=T_{z^{m-n}}T_{s^\sharp/q^\sharp}$ and $s^\sharp/q^\sharp\in \Rat(\BT)$. Now the second inclusion of \eqref{polyincl} gives
\[
\left(\frac{s^\sharp}{q^\sharp}\right) g= z^{n-m} k +\wtil{r}/q^\sharp,\quad \mbox{for some\ \  $\wtil{r}\in \cP_{n-1}$}.
\]
Write $\wtil{r}= \wtil{r}_1 q^\sharp + \wtil{r}_2$ with $\wtil{r}_2\in\cP_{m-1}$. Then $\wtil{r}/q^\sharp= \wtil{r}_1+ \wtil{r}_2/q^\sharp $ and $\deg(\wtil{r}_1)<m-n$. Since $\wtil{r}_2/q^\sharp\in \Rat_0(\BT)$ it follows that $g\in\Dom(T_{s^\sharp/q^\sharp})=\Dom(T_{\om^*})$ and $T_{s^\sharp/q^\sharp}g = z^{n-m} k +\wtil{r}_1$. But then $T_{\om^*}g= T_{z^{m-n}}T_{s^\sharp/q^\sharp}g= T_{z^{m-n}}(z^{n-m} k +\wtil{r}_1)=k$.
\end{proof}

A special case of the following result was proven as part of the proof of Theorem 2.2 in \cite{GtHJR2}.

\begin{lemma}\label{L:incl3}
Let $r,\wtil{r}\in\cP$ be co-prime. Then $r H^p\cap \wtil{r} H^p=r\wtil{r} H^p$.
\end{lemma}

\begin{proof}[\bf Proof]
Let $\wtil{r} f= rg$ with $f,g\in H^p$. Then $f=r \cdot g/\wtil{r}\in H^p$, so we should show $\wtil{f}:=g/\wtil{r}\in H^p$, i.e., $\wtil{f}$ analytic on $\BD$ and $\int_\BT |\wtil{f}(z)|^p\, dz<\infty$.

Since $g\in H^p$, the function $\wtil{f}$ can only fail to be analytic at the roots of $\wtil{r}$ inside $\BD$. However, if this were the case, then $f=r \wtil{f}$ would also fail to be analytic in $\BD$, since $r$ and $\wtil{r}$ are co-prime. Thus $\wtil{f}$ is analytic on $\BD$.

Divide $\BT$ as $\BT_1\cup \BT_2$ with $\BT_1\cap \BT_2=\emptyset$ in such a way that $\BT_1$ and $\BT_2$ are both nonempty finite unions of line segments of $\BT$ so that the interior of $\BT_1$ contains the roots of $r$ and the interior of $\BT_2$ the roots of $\wtil{r}$. Then $|\wtil{r}(z)|> N_1$ on $\BT_1$ and $|r(z)|> N_2$ on $\BT_2$ for some $N_1,N_2>0$. Note that $f=r \wtil{f}$ and $g=\wtil{r} \wtil{f}$. We then obtain
\begin{align*}
 \int_{\BT_2} |\wtil{f}(z)|^p\, dz
 & = \int_{\BT_2} |f(z)/r(z)|^p\, dz
 \leq N_2^{-p} \int_{\BT_2} |f(z)|^p\, dz
 \leq (2\pi N_2^{p})^{-1} \|f\|_{H^p}^p.
\end{align*}
Using $g=\wtil{r} \wtil{f}$, one obtains similarly that $\int_{\BT_1} |\wtil{f}(z)|^p\, dz\leq (2\pi N_1^{p})^{-1} \|g\|_{H^p}^p$. Thus $\int_\BT |\wtil{f}(z)|^p\, dz<\infty$.
\end{proof}

\begin{proof}[\bf Proof of Theorem \ref{T:main1'}] By Lemma \ref{L:incl1}, in order to prove \eqref{Tom*'}, the formula for the action of $T_\om^*$ on $q^\sharp H^{p'}$ and for the range of $T_{\om}^*$ in \eqref{Tom*RanKer'}, it remains to show that  $\Dom(T_\om^*)\subset q^\sharp H^{p'}$.

View $\cP$ and $\cP_k$, $k=1,2,\ldots$, as subspaces of $H^{p}$ or $H^{p'}$, write $P_k$ for the projection onto $\cP_{k-1}$ and set $Q_k=I-P_k$. Also, the standard $k\times k$ compression of a Toeplitz operator $T_{\phi}$ on $H^p$ (or $H^{p'}$) is denoted by $T_{\phi,k}$, i.e., $T_{\phi,k}=P_k T_{\phi}|_{\cP_{k-1}}$. Now, the relation \eqref{r1r} between $r\in\cP_{n-1}$ and $r_1\in\cP_{m-1}$ can be rewritten as
\[
T_s r_1-T_q r\in \cP_{m-1},
\]
or, equivalently, as
\begin{equation}\label{TsTrRel}
Q_mT_s P_m r_1 = Q_mT_s r_1 = Q_m T_q r = Q_m T_q P_n r.
\end{equation}
We now consider the cases $m\geq n$ and $m<n$ separately.

First assume $m\geq n$. We can then decompose $Q_mT_s P_m$ and $Q_m T_q P_n$ as
\begin{align*}
Q_mT_s P_m &=\mat{cc}{0& T_{s^\sharp,n}^* T_{z^{m-n}}^*\\ 0 & 0}:\cP_{m-1}=\mat{c}{\cP_{m-n}\\ T_{z^{m-n}} \cP_{n-1}}\to \mat{c}{\cP_{n-1}\\ T_{z}^{n}H^p},\\
Q_mT_q P_n &=\mat{c}{T_{q^\sharp,n}^*\\ 0}:\cP_{n-1}\to \mat{c}{\cP_{n-1}\\ T_{z^{n}}H^p}.
\end{align*}
Hence, in this case the identity in \eqref{TsTrRel} can be write as
\[
 T_{s^\sharp,n}^* (T_{z^{m-n}}^* r_1) = T_{q^\sharp,n}^* r.
\]
Since all Toeplitz matrices are upper triangular, we in fact have
\[
T_{s^\sharp,m}^* T_{z^{m-n},m}^* r_1=
T_{q^\sharp,m}^* r.
\]
Note that $T_{q^\sharp,n}^*$ is invertible, because $q$ has only roots on $\BT$ so that $q(0)\neq0$. We obtain that for given $r_1\in \cP_{m-1}$, the polynomial $r\in\cP_{n-1}$ that satisfies \eqref{r1r} is uniquely determined by
\begin{align*}
r&= (T_{q^\sharp,m}^*)^{-1} T_{s^\sharp,m}^* T_{z^{m-n},m}^* r_1
= T_{s^\sharp,m}^*T_{z,m}^{*m-n}(T_{q^\sharp,m}^*)^{-1} r_1,
\end{align*}
where the commutation of Toeplitz matrices can occur since they all have analytic symbols. Now take $r_1\in \cP_{m-1}$ arbitrary, and define $r$ as above, so that \eqref{r1r} holds. Then, by Lemma \ref{L:incl2}, we have
\begin{align*}
\inn{r_1}{P_m k}_{\cP_{m-1}}
&=\inn{r_1}{k}_{p,p'}=\inn{r}{g}_{p,p'} =\inn{r}{P_m g}_{\cP_{m-1}}\\
&=\inn{T_{s^\sharp,m}^*T_{z,m}^{*m-n}(T_{q^\sharp,m}^*)^{-1} r_1}{P_m g}_{\cP_{m-1}}\\
&=\inn{r_1}{(T_{q^\sharp,m})^{-1} T_{z,m}^{m-n} T_{s^\sharp,m} P_m g}_{\cP_{m-1}}.
\end{align*}
Since $r_1\in\cP_{m-1}$ is arbitrary, we obtain that $P_m k=(T_{q^\sharp,m})^{-1} T_{z,m}^{m-n} T_{s^\sharp,m} P_m g$, and thus
\[
P_m  T_{q^\sharp} k=
T_{q^\sharp,m} P_m k= T_{z,m}^{m-n} T_{s^\sharp,m} P_m g
=P_m T_{z}^{m-n} T_{s^\sharp} g.
\]
This shows that $P_m q^\sharp k= P_m z^{m-n} s^\sharp g$. Together with the first inclusion in \eqref{polyincl} we obtain that
\[
q^\sharp k= z^{m-n} s^\sharp g.
\]
Since $q^\sharp$ and $z^{m-n} s^\sharp$ are co-prime, we can apply Lemma \ref{L:incl3} to conclude  $g\in q^\sharp H^{p'}$.

Next assume $m<n$. We can then write $\om = \om_0 + \om_1$ uniquely with $\om_0\in\Rat_0(\BT)$ and $\om_1\in\Rat$ with no poles on $\BT$, i.e, $\om_1\in L^\infty (\BT)$, see
 \cite[Lemma 2.4]{GtHJR1}. In fact $\om_1\in\cP$, since all poles of $\om$ are on $\BT$, and $\om_0=\wtil{s}/q$ with $\wtil{s}\in\cP_{m-1}$. It now follows that $\Dom(T_{\om_0}^*)=q^\sharp H^{p'}$, and since $T_{\om_1}$ is bounded, $\Dom(T_\om^*)=\Dom(T_{\om_0}^*)=q^\sharp H^{p'}$. Furthermore, $T_{\om}^*= T_{\om_0}^*+ T_{\om_1}^*|_{q^\sharp H^{p'}}=T_{\om_0^*}|_{q^\sharp H^{p'}}+ T_{\om_1^*}|_{q^\sharp H^{p'}}=T_{\om^*}|_{q^\sharp H^{p'}}$.

In the next part of the proof we prove the formula for $\kernel(T_{\om^*})$, without distinguishing between the proper and non-proper case. Let $g=q^\sharp v\in\Dom(T_\om^*)$ with $v\in H^{p'}$.  Then $g\in \kernel(T_{\om}^*)$ if and only if $g\in \kernel(T_{\om^*})$, i.e., $g=q^\sharp v=r_1/(s_-)^\sharp$ for $r_1\in\cP_{\deg(s_-)-1}$, see Proposition \ref{P:Tom*}. Thus $v=r_1/((s_-)^\sharp q^\sharp)\in\Rat \cap H^{p'}$.  Then $v\in H^{p'}$  implies $r_1=q^\sharp r$, and $\deg(r)=\deg(r_1)-m<\deg(s_-)-m$. Hence $g=q^\sharp r/(s_-)^\sharp$ with $\deg(r)<\deg(s_-)-m$. That all such functions are in $\kernel(T_{\om}^*)=\kernel(T_{\om^*})\cap q^\sharp H^{p'}$ follows directly from the formula for $\kernel(T_{\om^*})$ obtained in Proposition \ref{P:Tom*}. The formula for the dimension of $\kernel(T_\om^*)$ follows directly and the condition for injectivity follows since $\deg(s_-)^\sharp$ is equal to the number on nonzero roots of $s_-$, counting multiplicity.
%
\end{proof}

\section{The adjoint of $T_{\om}$: General case}\label{S:general}

In the section we prove Theorem \ref{T:main1} in full generality. Hence let $\om=s/q\in \Rat$ with $s,q\in\cP$ co-prime. As in Theorem \ref{T:main1}, factor $s=s_-s_0s_+$ and $q=q_-q_0q_+$ with $s_-,q_-$ having roots only inside $\BT$, $s_0,q_0$ having roots only on $\BT$, and $s_+,q_+$ having roots only outside $\BT$. Set $m=\deg(q)$, $n=\deg(s)$,  $m_{\pm}=\degr(q_{\pm})$, $n_{\pm}=\degr(s_{\pm})$, and $m_{0}=\degr(q_{0})$, $n_{0}=\degr(s_{0})$. By Lemma 5.1 in \cite{GtHJR1}, and its proof, we can factor $\om$ as $\om=\om_- (z^\kappa \om_0) \om_+$ with $\kappa=n_- - m_-$, $\om_-=s_-/(z^{\kappa}q_-)$ having only poles and zeroes inside $\BT$, $\om_0=s_0/q_0$ having only poles and zeroes on $\BT$, and $\om_+=s_+/q_+$ having only poles and zeroes outside $\BT$, and we have $T_{\om}=T_{\om_-} T_{z^\kappa \om_0} T_{\om_+}$. Moreover, $T_{\om_-}$ and $T_{\om_+}$ are bounded and boundedly invertible.

Note that $T_{\omega_-}T_{z^\kappa\om_0}$ is closed and densely defined and ${\Ran}(T_{\om_+})=H^p$, and thus by Corollary 1 in \cite{Schechter70}
$$
T_\om^*=T_{\om_+}^* \left(T_{\om_-}T_{z^\kappa\om_0}\right)^*.
$$
Furthermore, $T_{\om_-}$ is bounded and $T_{z^\kappa\om_0}$ is closed and densely defined. By Theorem 4 in \cite{CasterenGoldberg70} one has
$$
 \left(T_{\om_-}T_{z^\kappa\om_0}\right)^* =
 T_{z^\kappa\om_0}^*T_{\om_-}^*.
$$
Combining this and using that $T_{\om_+}^*=T_{\om_+^*}$ and $T_{\om_-}^*=T_{\om_-^*}$ we see that
$$
T_\om^*=T_{\om_+}^*T_{z^\kappa\om_0}^*T_{\om_-}^*=T_{\om_+^*}T_{z^\kappa\om_0}^*T_{\om_-^*} \quad \mbox{on $\Dom(T_{\om}^*)$.}
$$
Note that
\[
\om_-^*=\frac{(s_-)^\sharp}{(q_-)^\sharp},\ \
\om_0^*=z^{m_0-n_0}\frac{(s_0)^\sharp}{(q_0)^\sharp},\ \
(z^\kappa\om_0)^*=z^{m_0-n_0-\kappa}\frac{(s_0)^\sharp}{(q_0)^\sharp},\ \
\om_+^*=z^{m_+-n_+}\frac{(s_+)^\sharp}{(q_+)^\sharp}.
\]

By construction, $\om_-$ and $1/\om_-$ are both anti-analytic. Consequently, $\om_-^*$ and $1/\om_-^*$ are both analytic functions. This implies $T_{\om_-^*}^{\pm} (q_0)^\sharp H^{p'}\subset (q_0)^\sharp H^{p'}$, and thus $T_{\om_-^*} (q_0)^\sharp H^{p'}= (q_0)^\sharp H^{p'}$. Since $T_{\om_+^*}$ is invertible, to see that $\Dom(T_\om^*)=(q_0)^\sharp H^{p'}$ it suffices to show $\Dom(T_{z^\kappa\om_0}^*)=(q_0)^\sharp H^{p'}$. For the case where $\kappa \geq 0$, so that $z^\kappa \om_0\in\Rat(\BT)$, this follows directly from Theorem \ref{T:main1'}. For $\kappa<0$, note that $T_{z^\kappa\om_0}=T_{z^\kappa}T_{\om_0}$, so that $T_{z^\kappa\om_0}^*=T_{\om_0}^*T_{z^\kappa}^*=T_{\om_0}^*T_{z^{-\kappa}}$, again using Theorem 4 of \cite{CasterenGoldberg70}. Then $g\in \Dom(T_{z^\kappa\om_0}^*)$ holds if and only if $z^{-\kappa} g\in \Dom(T_{\om_0}^*)=(q_0)^\sharp H^{p'}$. By Lemma \ref{L:incl3} this is the same as $g\in (q_0)^\sharp H^{p'}$, since $z^{-\kappa}$ and $q_0^\sharp$ are co-prime. Thus in both cases we arrive at $\Dom(T_\om^*)=(q_0)^\sharp H^{p'}$. Moreover, we also find that $T_{z^\kappa \om_0}^*=T_{(z^\kappa \om_0)^*}|_{(q_0)^\sharp H^{p'}}$, so that
\[
T_\om^*=T_{\om_+^*}T_{z^\kappa\om_0}^*T_{\om_-^*}=T_{\om_+^*}T_{(z^\kappa\om_0)^*}T_{\om_-^*}|_{(q_0)^\sharp H^{p'}}
=T_{\om^*}|_{(q_0)^\sharp H^{p'}}.
\]
Hence \eqref{Tom*} holds.

Next we derive the formula for $\kernel (T_{\om}^*)$. For $\kappa\geq0$ we have $g\in \kernel (T_{\om}^*)$ if and only if $T_{\om_-^*}g\in \kernel (T_{z^\kappa\om_0}^*)= (q_0)^\sharp \cP_{\kappa -m_0-1}$, where the last identity follows by applying Theorem \ref{T:main1'} to $z^\kappa \om_0$. Thus $g\in \kernel (T_{\om}^*)$ if and only if $((s_-)^\sharp/(q_-)^\sharp)g=(q_0)^\sharp r$, i.e., $g=(q_-)^\sharp (q_0)^\sharp r/(s_-)^\sharp$, for some $r\in \cP_{\kappa -m_0-1}$, as claimed. For $\kappa<0$ we have $g\in \kernel (T_{\om}^*)$ if and only if $z^{-\kappa} \om_-^* g\in \kernel (T_{\om_0}^*)$. However, $\kernel (T_{\om_0}^*)=\{0\}$, by Theorem \ref{T:main1'}, so that $\kernel (T_{\om}^*)=\{0\}$, in line with the formula in \eqref{Tom*RanKer}. The formula for the dimension of $\kernel(T_\om^*)$ follows directly.

Now we turn to the formula for $\Ran(T_\om^*)$. Note that
\begin{equation}\label{RanTom*}
 \Ran(T_\om^*)= T_{\om_+^*} \Ran (T_{z^\kappa \om_0}^* T_{\om_-^*})= T_{\om_+^*} \Ran (T_{z^\kappa \om_0}^*).
\end{equation}
We first show that $\Ran (T_{z^\kappa \om_0}^*)= T_{z^{m_0-n_0-\kappa}}(s_0)^\sharp H^{p'}$. Again, for the case $\kappa\geq 0$ this follows directly from Theorem \ref{T:main1'}. Assume $\kappa<0$. Then $T_{z^\kappa \om_0}^*=T_{\om_0}^* T_{z^{-\kappa}}$. Hence,
\begin{align*}
\Ran (T_{z^\kappa \om_0}^*)&=T_{\om_0}^* (z^{-\kappa} H^{p'}\cap \Dom(T_{\om_0}))
=T_{\om_0}^* (z^{-\kappa} H^{p'}\cap (q_0)^\sharp H^{p'})\\
&=T_{\om_0}^* z^{-\kappa} (q_0)^\sharp H^{p'}.
\end{align*}
The last identity follows by Lemma \ref{L:incl3}. Now the action of $T_{\om_0}^*$, as described in Theorem \ref{T:main1'}, shows that $\Ran (T_{z^\kappa \om_0}^*)=T_{z^{m_0-n_0}}z^{-\kappa}(s_0)^\sharp H^{p'}
=T_{z^{m_0-n_0-\kappa}}(s_0)^\sharp H^{p'}$. Since $1/q_+$ is analytic, $1/(q_+)^\sharp$ is anti-analytic, and therefore, independent of the sign of $m_+-n_+$, we have
\[
T_{\om_+^*}=T_{1/(q_+)^\sharp} T_{z^{m_+-n_+}} T_{(s_+)^\sharp}.
\]
Thus
\[
\Ran(T_{\om}^*)=T_{1/(q_+)^\sharp} T_{z^{m_+-n_+}} T_{(s_+)^\sharp} T_{z^{m_0-n_0-\kappa}}(s_0)^\sharp H^{p'}.
\]
Note that $T_{(s_+)^\sharp} $ and $ T_{z^{m_0-n_0-\kappa}}$ need not commute, in case $m_0-n_0-\kappa<0$. However, we do have $T_{(s_+)^\sharp}  T_{z^{m_0-n_0-\kappa}}=T_{z^{m_0-n_0-\kappa}} T_{(s_+)^\sharp} Q_{\kappa + n_0-m_0}$. Moreover, since $(s_+)^\sharp$ is analytic, $T_{(s_+)^\sharp} Q_{\kappa + n_0-m_0}=Q_{\kappa + n_0-m_0}T_{(s_+)^\sharp} Q_{\kappa + n_0-m_0}$ and we have
\[
T_{z^{m_+-n_+}} T_{z^{m_0-n_0-\kappa}} Q_{\kappa + n_0-m_0}
=T_{z^{m_+-n_+ + m_0-n_0-\kappa}} Q_{\kappa + n_0-m_0}
=T_{z^{m-n}} Q_{\kappa + n_0-m_0}.
\]
Therefore, we have
\begin{align*}
\Ran(T_{\om}^*)
&=T_{1/(q_+)^\sharp} T_{z^{m-n}} T_{(s_+)^\sharp}Q_{\kappa + n_0-m_0} (s_0)^\sharp H^{p'}\\
&=T_{z^{m-n}(s_+)^\sharp/(q_+)^\sharp}Q_{\kappa + n_0-m_0} (s_0)^\sharp H^{p'},
\end{align*}
again using that $1/(q_+)^\sharp$ is anti-analytic and $(s_+)^\sharp$ is analytic. This gives the general formula for $\Ran(T_\om^*)$. In case $\kappa + n_0-m_0\leq 0$, we have $Q_{\kappa + n_0-m_0}=I$ and $T_{(s_+)^\sharp}Q_{\kappa + n_0-m_0} (s_0)^\sharp= (s_+s_0)^\sharp$, as claimed.
%

\section{Symmetric operators and selfadjoint extensions}\label{S:symmetric}

For $\om \in\Rat$, the  second adjoint $T_\om^{**}$ is well-defined and $T_\om^{**} = T_\om$, since $T_\om$ is a closed, densely defined operator on a reflexive Banach space \cite[Theorem III.5.24]{K95}. Now consider $\om\in \Rat(\BT)$ and $p=2$. From Theorem \ref{T:main1} it is obvious that $T_{\om}\neq T_{\om}^*$, except in the degenerate case where $q$ is constant, since $\Dom(T_\om)= q H^2 +\cP_{\deg(q)-1}$ contains all polynomials while $\Dom(T_\om^*)= q^\sharp H^2$ only contains the polynomials that contain $q^\sharp$ as a factor. Consequently, $T_{\om}$ cannot be selfadjoint. In this section we consider the question when $T_\om^*$ is symmetric, and, if this is the case, when does $T_\om^*$ have a selfadjoint extension $L$. The first topic is addressed in the following theorem.

\begin{theorem}\label{T:symm}
Let $\om=s/q\in\Rat(\BT)$ with $s,q\in\cP$ co-prime. Set $n=\deg(s)$ and $m=\deg(q)$. Then the following are equivalent.
\begin{itemize}
  \item[(1)] $T_\om^*$ is symmetric;

  \item[(2)] $\om(\BT)\subset\BR$;

  \item[(3)] $\om(z)=\wtil{\om}(-i\frac{z+1}{z-1})$ with $\wtil{\om}$ a real rational function with poles only on $\BR$;

  \item[(4)] the essential spectrum $\si_\tu{ess}(T_\om)$ of $T_\om$ is contained in $\BR$;

  \item[(5)] $\om$ is proper,  $s=z^{m-n} \wtil{s}$ with $\wtil{s}$ self-inversive and $q_0 \overline{s_n}=\overline{q_m} s_{m-n}$ holds, where $s(z)=\sum_{k=0}^{n} s_k z^k$ and $q(z)=\sum_{k=0}^{m} q_k z^k$.

\end{itemize}
Moreover, if $T_\om^*$ is symmetric, then $T_\om^*\subset T_\om$.
\end{theorem}

\begin{proof}[\bf Proof]
We first prove the equivalence of (1) and (2), and that (1) implies $T_\om^*\subset T_\om$. Assume (2). Then, for $z\in\BT$, not a root of $q$, we have $\om^*(z)=\overline{\om(z)}=\om(z)$. Hence $\om^*=\om$. Since $q$ has only roots on $\BT$, we have $q=\gamma q^\sharp$ for a unimodular constant $\ga$. Hence $q H^2=q^\sharp H^2$. This shows $T^*_{\om}=T_{\om^*}|_{q^\sharp H^2}=T_{\om}|_{q H^2}\subset T_\om$. Since $(T_\om^{*})^*=T_\om$, it follows that $T_\om^*$ is symmetric and $T_\om^*\subset T_\om$. Conversely, assume (1). Then we still have $q H^2=q^\sharp H^2$ and $T_{\om}^*\subset (T_\om^*)^*=T_\om$. Hence $T_\om^*=T_\om|_{q H^2}$. In particular, we have $\om^* q=T_{\om^*} q=T_\om^* q=T_\om q= \om q$. This implies $\om =\om^*$. Hence $\om(z)=\overline{\om(z)}$ for $z\in\BT$, not a root of $q$. Thus $\om(\BT)\subset \BR$.

That (2) and (3) are equivalent follows simply because in (3) $\om$ is the composition of $\wtil{\om}$ and the inverse Cayley transform, which maps the circle $\BT$ bijectively onto $\BR$. The fact that $\wtil{\om}$ is real rational, i.e., $\wtil{\om}=\wtil{s}/\wtil{q}$ with $\wtil{s}$ and $\wtil{q}$ real polynomials, is equivalent to $\wtil{\om}(\BR):=\{\wtil{\om}(t) \colon t\in\BR,\ \wtil{q}(t)\neq 0\}\subset \BR$. Also, the equivalence of (2) and (4) is a direct consequence of the fact that $\si_\tu{ess}(T_\om)=\om(\BT)$, by \cite[Theorem 1.1]{GtHJR2}.

Finally, we prove (2) $\Leftrightarrow$ (5). Since $q=\ga q^\sharp$, we have
\[
\om^*=z^{m-n} \frac{s^\sharp}{q^\sharp}= z^{m-n} \ga \frac{s^\sharp}{q}.
\]
Thus, we have $\om=\om^*$ if and only if $z^{m-n} \ga s^\sharp= s$. Hence (2) is equivalent to $z^{m-n}\ga s^\sharp= s$. Now assume (2). Since $\deg(s^\sharp)\leq \deg(s)$, the identity $z^{m-n}\ga s^\sharp= s$ can only occur if $m\geq n$, i.e., if $\om$ is proper. The identity also shows that $s= z^{m-n} \wtil{s}$ for $\wtil{s}=\ga s^\sharp$. On the other hand, $s^\sharp=(z^{m-n} \wtil{s})^\sharp= \wtil{s}^\sharp$. Thus $ \wtil{s}= \ga s^\sharp= \ga \wtil{s}^\sharp$, which shows $\wtil{s}$ is self-inversive, with constant $\ga$. Note that $\ga=q_0/\overline{q_m}$. Also, we have $s_0=\cdots=s_{m-n-1}=0$ and  $\wtil{s}(z)=\sum_{k=0}^{2n-m} s_{m-n+k} z^k$. Since $\wtil{s}$ is self-inversive, $\wtil{s}=\de \wtil{s}^\sharp$ with $\de=s_{m-n}/\overline{s_n}$. But also $\de=\ga$, so $s_{m-n}/\overline{s_n}=q_0/\overline{q_m}$. Thus $q_0 \overline{s_n}=\overline{q_m} s_{m-n}$. Hence (5) holds. Conversely, assume (5). Reversing the above argument, it follows that $q_0 \overline{s_n}=\overline{q_m} s_{m-n}$ implies $\wtil{s}=\de \wtil{s}^\sharp$ with $\de=\ga$. Thus $\ga s^\sharp=\ga \wtil{s}^\sharp = \wtil{s}$. This implies $ s =z^{m-n} \wtil{s}= z^{m-n} \ga s^\sharp$, and hence (2).
\end{proof}

\begin{corollary}\label{C:degcons}
Let $\om=s/q\in\Rat(\BT)$ with $s,q\in\cP$ co-prime.Assume $T_\om^*$ is symmetric. Then $\deg(s)\leq \deg(q)\leq 2\deg(s)$.
\end{corollary}

\begin{proof}[\bf Proof]
By Theorem \ref{T:symm} condition (5) holds with $m=\deg(q)$ and $n=\deg(s)$. Since $\wtil{s}$ is self-inversive, we have $\wtil{s}(0)\neq 0$. Consequently, 0 would be a non-removable singularity of $s=z^{m-n}\wtil{s}$ in case $m<n$, which gives a contradiction. Hence $m\geq n$. Furthermore, comparing the degrees on both sides of $s=z^{m-n}\wtil{s}$ yields, $n=m-n+\deg(\wtil{s})\geq m-n$. Hence $m\leq 2n$.
\end{proof}

When $T_{\om}^*$ is symmetric, it need not be the case that $T_{\om}^*$ has a selfadjoint extension. In Proposition \ref{P:selfadjext} below we characterize when $T_{\om}^*$ does have a selfadjoint extension. However, we first give a concrete example that shows this does not always happen.

\begin{example}\label{E:symmetric1}
In \cite{Helson} Helson considered the functions $\om_k(z)=\left(-i\frac{z+1}{z-1}\right)^k$ for $k\in \BN$. For all  $k$ we have $\om_k(\BT)\subset \BR$, see Theorem \ref{T:symm} (3) above, hence $T_{\om_k}^*$ is symmetric by Theorem \ref{T:symm}. In fact, for $k$ even $\om_k(\BT)=\BR_+$, while for $k$ odd we have $\om_k(\BT)=\BR$. We show that $T_{\om_k}^*$ does not have a selfadjoint extension for $k=1$. In Example \ref{E:symmetric2} we return to this example for general $k$.

For $k=1$ we have $\om(z)=\om_1(z) = -i\frac{z+1}{z-1}$. Hence $\Dom(T_\om) = (z-1)H^2 + \BC$ and $\Dom(T_\om^*) = (z-1)H^2$. Suppose $T_\om^*$ has a selfadjoint extension $L$. Then $L=L^*$ and thus $T_{\om}^*\subset L=L^* \subset T_{\om}^{**}=T_\om$. Since $T_\om$ is not selfadjoint, the inclusions are strict. Hence $\Dom(T_\om^*)\subset \Dom(L)\subset \Dom(T_\om)$, with strict inclusions. However, the complement of $\Dom(T_\om^*)$ in $\Dom(T_\om)$ is one-dimensional, hence not both inclusions can be strict. Thus $T_\om$ does not admit a selfadjoint extension.
\end{example}

\begin{proposition}\label{P:selfadjext}
Let $\om=s/q\in\Rat(\BT)$, with $s,q\in\cP$ coprime, be such that $T_\om^*$ is symmetric. Then $T_\om^*$ admits a selfadjoint extension if and only if the number of roots of $s-iq$ and $s+iq$ in $\BD$, counting multiplicities, coincide.
\end{proposition}

\begin{proof}[\bf Proof]
The operator $T_\om^*$ is an adjoint, and hence closed, and by assumption symmetric. Following definition X.2.12 from \cite{Conway} we define the deficiency subspaces of $T_\om^*$ as the spaces
\[
\mathcal{L}_+ ={\rm Ker\, } (T_\om^{**}-i)=({\rm Ran\, }(T_\om^*+i))^\perp,
\quad
\mathcal{L}_- ={\rm Ker\, } (T_\om^{**}+i)=({\rm Ran\, }(T_\om^*-i))^\perp,
\]
and the deficiency indices as the integers $n_{\pm}={\rm dim\, }\mathcal{L}_{\pm}$. Since $T_{\om}^{**}=T_\om$, we have
\[
n_+={\rm dim\, }{\rm Ker\, } (T_\om-i)
\ands
n_-={\rm dim\, }{\rm Ker\, } (T_\om+i).
\]
Also, we have $T_\om\pm i= T_{\om \pm i}$. By item (b) of Theorem X.2.20 in \cite{Conway}, $T_\om$ has a selfadjoint extension if and only if $n_+=n_-$. Note that $\om \pm i = (s\pm i q)/q$. We now apply Corollary 4.2 from \cite{GtHJR1} to $T_{\om\pm i}$, to obtain that $n_\pm$ is equal to the maximum of 0 and the difference of $m$ and the number of roots of $s\pm i q$ in $\overline{\BD}$, counting multiplicities. However, since $T_\om^*$ is symmetric, $\om$ is proper so the number of roots cannot exceed $m$. Note also that $\om(\BT)\subset \BR$, so $s\pm i q$ cannot have roots on $\BT$. It thus follows that $T_\om^*$ has a selfadjoint extension if and only if the number of roots in $\BD$ of
$s-iq$ and $s+iq$, counting multiplicities, coincide, as claimed.
\end{proof}

Since $T_\om^*$ is never selfadjoint for $\om\in\Rat(\BT)$ having at least one pole on $\BT$, the formulas for $n_\pm$ in the above proof along with item (a) of Theorem X.2.20 in \cite{Conway} directly give the following corollary.

\begin{corollary}\label{C:nonselfad}
Let $\om=s/q\in\Rat(\BT)$, with $s,q\in\cP$ coprime, be such that $T_\om^*$ is symmetric. Then $s+i q$ or $s-iq$ must have a root in $\BD$.
\end{corollary}

Proposition \ref{P:selfadjext} can be rephrased in terms of the index of the operators $T_{\om\pm i}$.

\begin{proposition}\label{P:selfadjext2}
Let $\om=s/q\in\Rat(\BT)$, with $s,q\in\cP$ coprime, be such that $T_\om^*$ is symmetric. Then $T_{\om +i}$ and $T_{\om- i}$ are both Fredholm and $T_\om^*$ admits a selfadjoint extension if and only if the Fredholm indices of $T_{\om +i}$ and $T_{\om- i}$ coincide.
\end{proposition}

\begin{proof}[\bf Proof]
This follows directly from Proposition \ref{P:selfadjext} and Theorem 1.1 of \cite{GtHJR1} applied to $\om+ i$ and $\om - i$, using that $\om\pm i=(s\pm i q)/q$.
\end{proof}

\begin{corollary}\label{C:notR}
Let $\om=s/q\in\Rat(\BT)$, with $s,q\in\cP$ coprime, be such that $T_\om^*$ is symmetric. Assume $\om(\BT)\neq \BR$. Then $T_\om^*$ admits a selfadjoint extension.
\end{corollary}

\begin{proof}[\bf Proof]
The Fredholm index of $T_{\om -\la}$ is constant with respect to $\la\in\BC$ on the connected components of $\BC$ separated by the essential spectrum of $T_\om$, which is equal to $\om(\BT)$; see \cite[Theorem 1.1]{GtHJR2}. Hence if $\om(\BT)\neq \BR$, but $\om(\BT)\subset\BR$ since $T_\om^*$ is symmetric, then $i$ and $-i$ are in the same connected component and thus $T_{\om+i}$ and $T_{\om-i}$ have the same index. The conclusion now follows from Proposition \ref{P:selfadjext2}.
\end{proof}

\begin{example}\label{E:symmetric2}
We return to the functions $\om_k(z)=\left(-i\frac{z+1}{z-1}\right)^k$ considered in Example \ref{E:symmetric1}. Since $\om_k(\BT)=\BR_+$ for $k$ even, we obtain directly from Corollary \ref{C:notR} that $T_{\om_k}^*$ admits a selfadjoint extension in case $k$ is even.

For odd values of $k$ we have $\om_k(\BT)=\BR$, and thus no conclusion can be drawn from Corollary \ref{C:notR}. To deal with the odd case we resort to Proposition \ref{P:selfadjext}.
Take $s(z)=(-i)^k(z+1)^k$ and $q=(z-1)^k$ and write $k$ as $k=2l+1$. The polynomials
$s\pm iq$ are given by
\begin{align*}
s(z)\pm iq(z)&=i\left( (-1)^{l+1}(z+1)^{2l+1}\pm (z-1)^{2l+1}\right) \\
&=i\left(  (-1)^{l+1}\sum_{j=0}^{2l+1}\binom{2l+1}{j}z^j
\pm \sum_{j=0}^{2l+1} \binom{2l+1}{j}z^j (-1)^{2l+1-j}\right) \\
&=i\sum_{j=0}^{2l+1} \binom{2l+1}{j}z^j \left((-1)^{l+1}\pm(-1)^{2l+1-j}\right)\\
&=i \sum_{j=0}^{2l+1}\binom{2l+1}{j}z^j \left((-1)^{l+1}\pm(-1)^{j-1}\right).
\end{align*}
For odd values of $l$ one obtains:
\begin{align*}
s(z)-iq(z)&= -2i \left(
\binom{2l+1}{0} +\cdots + \binom{2l+1}{2l-2}z^{2l-2} + \binom{2l+1}{2l}z^{2l}
\right),\\
s(z)+iq(z)&= 2i \left(
\binom{2l+1}{1}z +\cdots + \binom{2l+1}{2l-1}z^{2l-1} + \binom{2l+1}{2l+1}z^{2l+1}
\right)\\
&= 2i z \left(
\binom{2l+1}{2l} +\cdots + \binom{2l+1}{2}z^{2-2} + \binom{2l+1}{0}z^{2l}
\right)
\end{align*}
Observe that $s+iq$ is of the form $izp_+(z^2)$ where $p_+$ is a real polynomial of degree $2l$ and that $s-ig$ is of the form $i p_-(z^2)$ where $p_-$ is a real polynomial of degree $2l$. Because $p_+$ and $p_-$ are real polynomials and the fact that $z^2$ is the variable rather than $z$ itself, the nonzero roots of $zp_+(z^2)$ come either in pairs ($z$ and $-z$) for real nonzero roots or in quadruples ($z, \bar{z}, -z, -\bar{z}$) for nonreal roots, while zero appears as a simple root. Similarly, the roots of $p_-(z^2)$ come in pairs ($z$ and $-z$) or quadruples ($z, \bar{z}, -z, -\bar{z}$) and there is no root at zero.  Hence $s+iq$ has an odd number of roots inside the unit disc, and $s-iq$ has an even number of roots inside the unit disc, so that the indices $n_+$ and $n_-$ can never coincide. One further observes that $p_-=p_+^\sharp$. In a similar way, for even values of $l$ the polynomial $s+iq$ will have an even number of roots inside the unit disc and $s-iq$ will have an odd number of roots inside the unit disc. Hence, in all cases where $k$ is odd, $T_\om^*$ does not have a selfadjoint extension.
\end{example}

We now present a proposition that rephrases the criteria of Proposition \ref{P:selfadjext} in terms of the roots of $s+iq$ (or $s-iq$) only. The observation that $T_{\om_k}^*$ in Example \ref{E:symmetric2} has no selfadjoint extension follows as a special case. In general, $T_\om^*$ cannot have a selfadjoint extension whenever $\deg(q)$ is odd for any $\om\in\Rat(\BT)$.

\begin{proposition}\label{P:selfadjext3}
Let $\om=s/q\in\Rat(\BT)$, with $s,q\in\cP$ coprime, be such that $T_\om^*$ is symmetric. Set $l_{\pm}=m-\deg(s\pm iq)$ and define
\[
k_{\pm,1}=\# \left\{\begin{array}{l}\!\!\!
 \textrm{zeroes of } \om\pm i \textrm{ inside }\BT \!\!\! \\
\!\!\!\textrm{multi. taken into account}\!\!\!
\end{array}\right\},\quad
k_{\pm,2}=\# \left\{\begin{array}{l}\!\!\!
 \textrm{zeroes of } \om\pm i \textrm{ outside }\BT \!\!\! \\
\!\!\!\textrm{multi. taken into account}\!\!\!
\end{array}\right\},
\]
Then
\[
\mbox{$T_\om^*$ has a selfadjoint extension}
\quad \Longleftrightarrow \quad
l_+ + k_{+,2}= k_{+,1}
\quad \Longleftrightarrow \quad
l_- + k_{-,2}= k_{-,1}.
\]
In particular, if $T_\om^*$ has a selfadjoint extension, then $\deg(q)$ must be even.
\end{proposition}

%

The basis for the proof of  Proposition \ref{P:selfadjext3} lies in the following lemma, which
clarifies the relation between $s+iq$ and $s-iq$ under the assumption that $T_\om^*$ is symmetric.

\begin{lemma}\label{L:ids}
Let $\om=s/q\in\Rat(\BT)$, with $s,q\in\cP$ coprime, be such that $T_\om^*$ is symmetric. Set $l_{\pm}=\deg(q)-\deg(s\pm iq)$ and let $\ga$ be the unimodular constant such that $q=\ga q^\sharp$. Then
\begin{equation}\label{ids}
s\pm iq=\ga z^{l_{\mp}}(s\mp iq)^\sharp.
\end{equation}
Moreover, we have $l_\pm=0$ if and only if $\om(0)=\pm i$. In particular, only one of $l_+$ and $l_-$ can be nonzero.
\end{lemma}

\begin{proof}[\bf Proof]
Since $T_\om^*$ is symmetric, by assumption, $\om$ has the properties listed in Theorem \ref{T:symm}. In particular, $\om$ is proper, $m:=\deg(q)\geq \deg(s)=:n$, and $s=z^{m-n} \wtil{s}$ with $\wtil{s}$ self-inversive and the unimodular constants that establish the self-inversiveness of $\wtil{s}$ and $q$ coincide (equivalently, $q_0 \overline{s_n}=\overline{q_m}s_{m-n}$).

Note that $\deg(s\pm iq)\neq m$ occurs precisely when $\deg(s)=\deg(q)$ and the leading coefficients $s_m$ and $q_m$ of $s$ and $q$, respectively, satisfy $s_m\pm i q_m=0$, i.e., $s_m/q_m= \mp i$. Since $m=n$, the identity $q_0 \overline{s_n}=\overline{q_m}s_{m-n}$ shows
$\om(0)=s_0/q_0=\overline{s_m}/\overline{q_m}$. Hence $\deg(s\pm iq)\neq m$ holds if and only if
$\om(0)=\overline{\mp i}=\pm i$, as claimed.

We first prove \eqref{ids} for the case $\om(0)=0$. So assume $\om(0)=0$, or equivalently, $s(0)=0$. In this case $l_+=l_-=0$. Since $s=z^{m-n}\wtil{s}$ and $\wtil{s}(0)\neq 0$ (because $\wtil{s}$ is self-inversive), we have $m>n$. Also note that $m-n$ is equal to the multiplicity of 0 as a root of $s$. We now employ Lemma \ref{L:sumpound}, using that $\deg(s+iq)=m=\deg(iq)$, to obtain
\begin{align*}
\ga (s\mp iq)^{\sharp} & = z^{\deg(s+iq)-\deg(s)} \ga  s^\sharp \mp (-i) \ga  q^\sharp
=z^{m-n} \ga \wtil{s}^\sharp \pm i q
= z^{m-n} \wtil{s} \pm i q
=s\pm iq.
\end{align*}
Hence \eqref{ids} holds.

Now assume $\om(0)\neq 0$, i.e., $s(0)\neq 0$. In that case $s=\wtil{s}$. Hence $s$ is self-inversive with the same constant $\ga$ that establishes the self-inversiveness of $q$. This also yields $m=n$. Since $s$ and $q$ are self-inversive with the same constant $\ga$, we have
\[
\overline{s_{m-k}} q_{k} =\overline{q_{m-k}} \overline{s_{m-k}} \ga = \overline{q_{m-k}} s_{k}\quad \mbox{ for $k=0,\ldots, m$}.
\]
Hence for all $k$ we have
\[
\overline{s_{m-k}}(s_k+i q_k)= s_k (\overline{s_{m-k}} + i \overline{q_{m-k}})
\ands
\overline{q_{m-k}}(s_k+i q_k)= q_k (\overline{s_{m-k}} + i \overline{q_{m-k}}).
\]
In case $s_{m-k}=0$ and $q_{m-k}=0$, also $s_k=0$ and $q_k=0$, since $s_k=\ga \overline{s_{m-k}}$ and $q_k=\ga \overline{q_{m-k}}$, and thus $s_k+i q_k=0=\ga (\overline{s_{m-k}} + i \overline{q_{m-k}})$. If either $s_{m-k}\neq 0$ or $q_{m-k}\neq 0$, divide the first identity by $\overline{s_{m-k}}$ or the second identity by $\overline{q_{m-k}}$  to arrive at  $s_k+i q_k=\ga (\overline{s_{m-k}} + i \overline{q_{m-k}})$. Hence
\begin{equation}\label{ids2}
s_k+i q_k=\ga (\overline{s_{m-k} - i q_{m-k}}) \quad \mbox{ for $k=0,\ldots, m$}.
\end{equation}
In particular, $s_k+i q_k=0$ if and only if $s_{m-k} - i q_{m-k}=0$. It follows that $0$ is a root of $s\pm iq$ with multiplicity $l_\mp$. Comparing coefficients, it follows that the identities in \eqref{ids} correspond to the identities in \eqref{ids2}. Hence \eqref{ids} holds.
\end{proof}

\begin{proof}[\bf Proof of Proposition \ref{P:selfadjext3}]
Since $T_\om^*$ is assumed to be symmetric, \eqref{ids} holds. Together with the fact that the $\sharp$ operator reflects roots over $\BT$, this implies that the number of roots of $s\pm iq$ inside $\BT$ are equal to $l_\pm$ plus the number of roots of $s\mp iq$ outside $\BT$, counting multiplicities. In other words, we have
\begin{equation}\label{ks-ids}
k_{+,1}=l_- + k_{-,2} \ands k_{-,1}=l_+ + k_{+,2}.
\end{equation}
By Proposition \ref{P:selfadjext2}, $T_{\om}^*$ has a selfadjoint extension if and only if $s+iq$ and $s-iq$ have an equal number of roots inside $\BT$, again counting multiplicities, equivalently, $k_{+,1}=k_{-,1}$. Given \eqref{ks-ids}, it follows that $k_{+,1}=k_{-,1}$ is equivalent to $k_{+,1}=l_+ + k_{+,2}$, and likewise to $k_{-,1}=l_- + k_{-,2}$. This proves the two criteria for $T_{\om}^*$ to have a selfadjoint extension.

By Lemma \ref{L:ids}, either $l_+=0$ or $l_-=0$. Say $l_+=0$. Since $s+iq$ cannot have roots on $\BT$, we have $\deg(q)=\deg(s+iq)=k_{+,1}+k_{+,2}$. If $T_\om^*$ admits a selfadjoint extension, then we have $k_{+,1}=l_+ + k_{+,2}=k_{+,2}$. Hence $\deg(q)= 2 k_{+,1}$ is even. For $l_-=0$ the arguments goes similarly.
\end{proof}

Combining the fact that $T_\om^*$ cannot have a selfadjoint extension in case $\om=s/q\in\Rat(\BT)$, $s,q$ co-prime, and $\deg(q)$ odd with Corollary \ref{C:notR} immediately yields the following result.

\begin{corollary}\label{C:odd-line}
Let $\om=s/q\in\Rat(\BT)$, with $s,q\in\cP$ co-prime, be such that $T_\om^*$ is symmetric and $\deg(q)$ is odd. Then $\om(\BT)=\BR$.
\end{corollary}

The next example shows that also with $\deg(q)$ even it can occur that $T_\om^*$ does not admit a selfadjoint extension.

\begin{example}\label{E:nonselfadjeven}
Let $\om=s/q$ with
\[
s(z)=i(1+az+z^2),\ \mbox{ for some $0\neq a\in\BR$ }, \ands q(z)=1-z^2.
\]
Then $m=n$ and
\[
s^\sharp =-s,\quad q^\sharp =-q.
\]
So $T_\om^*$ is symmetric by Theorem \ref{T:symm} (5). Also, we have
\[
(s+iq)(z)=i(2+az)\ands (s-iq)(z)=iz(a+2z).
\]
Hence the number of roots of $s-iq$ inside $\BD$ is 1 if $|a|\geq 2$ and 2 if $0\neq |a|< 2$, while the number of roots of $s+iq$ inside $\BD$ is 1 if $|a|> 2$ and 0 if $0\neq|a|\leq 2$. Thus $T_\om^*$ admits a selfadjoint extension if and only if $|a|>2$.
\end{example}

\section{Comparison with the unbounded Toeplitz operator defined by  Sarason}\label{S:Sarason}

The Smirnov class $N^+$ consists of quotients $\frac{b}{a}$ with $a$ and $b$ $H^\infty$-functions such that the denominator $a$ is an outer function. The function $\vph = \frac{b}{a}\in N^+$ is said to be in \emph{canonical form} if $a(0) > 0$ and $\vert a\vert^2 + \vert b\vert^2 = 1$ on $\BT$. By Proposition 3.1 of \cite{S08}, every function $\vph\in N^+$ can be uniquely written in canonical form.

In \cite{S08}, Sarason investigated an unbounded Toeplitz operator $T_{\vph}^\tu{Sa}$ with symbol $\vph$ in $N^+$, which is defined by
\[
\Dom (T_\vph^\tu{Sa}) = \{ f\in H^2 : \varphi f \in H^2 \},\quad  T_\vph^\tu{Sa} f = \varphi f\ (f\in \Dom (T_\vph^{Sa})).
\]
More generally, $T_\vph^\tu{Sa}$ can be defined in this way for any holomorphic function $\vph$ on $\BD$, but for $T_\vph^\tu{Sa}$ to be densely defined, $\vph$ must be in $N^+$; see \cite[Lemma 5.2]{S08}.

Let $\varphi = \frac{b}{a}\in N^+$ be the canonical representation of $\vph$. Then it is shown in  Proposition 5.3 of \cite{S08} that $\Dom(T_\vph^\tu{Sa}) = aH^2$. The adjoint of the operator $T_\vph^\tu{Sa}$ is motivated by the action of the conjugate transpose of the matrix representation of $T_\vph^\tu{Sa}$, which is lower triangular. The domain of the adjoint operator is shown to contain the space $H(\overline{\BD})$ of functions that are analytic on some neighborhood of the closed unit disc $\overline{\BD}$, and the adjoint is equal to the closure of the operator on $H(\overline{\BD})$; see \cite[Lemmas 6.1 and 6.4]{S08}.

Let $\om=s/q\in\Rat(\BT)$ with $s,q\in\cP$ co-prime. Set $n=\degr(s)$ and $m = \degr(q)$. Assume $\om$ is proper, i.e., $n\leq m$. Then $\om^* (z) = z^{m-n}s^\sharp/q^\sharp\in\Rat(\BT)$. Since $q^\sharp$ has zeroes only on $\BT$ it is outer and thus $\om^* \in N^+$. While in general $T_\om$ and $T_\om^\tu{Sa}$ are different, the following proposition shows that $T_\om$ coincides with $T_{\om^*}^\tu{Sa}$, and hence $T_\om=T_\om^{**}=T_{\om^*}^\tu{Sa}$. Without the properness assumption, $\om^*$ is not in $N^+$, because $\om^*$ has a pole at 0, and hence $T_{\om^*}^\tu{Sa}$ is not defined.

\begin{proposition}\label{P:SaToep}
Let $\wtil{\om}=\wtil{s}/\wtil{q}\in\Rat(\BT)$ with $\wtil{s},\wtil{q}\in\cP$ co-prime. Then $\Dom(T_\wtil{\om}^\tu{Sa})=\wtil{q}H^2$ and $T_\wtil{\om}^\tu{Sa}=T_{\wtil{\om}}|_{\wtil{q}H^2}$. In particular, if $\om\in \Rat(\BT)$ is proper, then $T_{\om}^*=T_{\om^*}^\tu{Sa}$.
\end{proposition}

\begin{proof}[\bf Proof]
We first show $\Dom(T_{\wtil{\om}}^\tu{Sa})=\wtil{q}H^2$. Let $\wtil{\om}=a/b$ be the canonical form of $\wtil{\om}$. As noted above, $\Dom(T_{\wtil{\om}}^\tu{Sa})=a H^2$. By the Fej\'er-Riesz Theorem there is a polynomial $r$ such that on $\BT$ we have $\vert r\vert^2 = \vert \wtil{s}\vert^2 + \vert \wtil{q}\vert^2$, $r$ has no roots in $\BD$ and $\tu{arg}(r(0))=\tu{arg}(\wtil{q}(0))$. The latter is possible since $\wtil{q}(0)\neq 0$ and implies $\wtil{q}(0)/r(0)>0$. Note that $r$ also has no roots on $\BT$, since $\wtil{s}$ and $\wtil{q}$ are co-prime. It follows that $\wtil{q}/r$ and $\wtil{s}/r$ are both $H^\infty$-functions, $\wtil{q}/r$ is outer and $\wtil{q}(0)/r(0)>0$. Hence $a=\wtil{q}/r$ and $b=\wtil{s}/r$, by the uniqueness of the canonical form. Also, since all the roots of $r$ are outside $\BT$, $r^{-1} H^2=H^2$, so that $a H^2= \wtil{q} H^2$.

Now let $f\in \Dom(T_\wtil{\om}^\tu{Sa})$, say $f= \wtil{q}h$ with $h\in H^2$. Then $T_\wtil{\om}^\tu{Sa} f=\wtil{\om} f= \wtil{s} h$. On the other hand, the fact that $\wtil{\om} f= \wtil{s} h$ and $\wtil{s} h\in H^2$ shows $T_\wtil{\om} f =\BP \wtil{s} h=\wtil{s} h$. Hence  $T_\wtil{\om}^\tu{Sa}=T_{\wtil{\om}}|_{\wtil{q}H^2}$.
\end{proof}

Next we employ some of the ideas from \cite{S08} to derive the following result. Recall that for a Hilbert space operator $T:\Dom(T)\to\cH$ a linear submanifold $\cD\subset \Dom(T)$ is called a {\em core} in case the graph $G(T|_{\cD})$ of $T|_{\cD}$ is dense in the graph $G(T)$ of $T$; cf., page 166 in \cite{K95}.

\begin{theorem}\label{T:core}
Let $\om\in\Rat(\BT)$. Then $H(\overline{\BD})$ is contained in $\Dom(T_\om)$. If $\om$ is proper, then $H(\overline{\BD})$ is a core of $T_\om$.
\end{theorem}

\begin{proof}[\bf Proof of $H(\overline{\BD})\subset \Dom(T_\om)$]
Write $\om=\frac{s}{q}\in\Rat_0(\BT)$ with $s,q\in\cP$ coprime. Let $f\in H(\overline{\BD})$. Then there exists a $R>1$ such that $f$ is still analytic on an open neighborhood of the closed disc with radius $R$. Set $\wtil{f}(z)=f(Rz)$, $\wtil{q}(z)=q(Rz)$ and $\wtil{s}(z)=s(Rz)$. Then $\wtil{f}\in H^2$ and $\wtil{q}$ is a polynomial with no roots on $\BT$ and $\deg(q)=\deg(\wtil{q})$. By Theorem 3.1 in \cite{GtHJR1}, $H^2=\wtil{q} H^2+ \cP_{\deg(q)-1}$. Thus $\wtil{s}\wtil{f}=\wtil{q}\wtil{h}+\wtil{r}$ for some $\wtil{h}\in H^2$ and $\wtil{r}\in\cP$ with $\deg(\wtil{r})<\deg(q)$. Now set $r(z)=\wtil{r}(z/R)$ and $h(z)=\wtil{h}(z/R)$. Then $r\in\cP$ with $\deg(r)=\deg(\wtil{r})<\deg(q)$ and $h\in H^2$, even $h\in H(\overline\BD)$. Also, we have $s f =qh +r$. Thus $f\in\Dom(T_\om)$.
\end{proof}

Before proving the second claim of Theorem \ref{T:core} it is useful to consider the value of $T_\om$ when applied to the evaluation functional or reproducing kernel element $k_\lambda (z) = (1 - \overline{\lambda}z)^{-1}$, where $\la\in\BD$. Note that $k_\la\in H(\overline{\BD})$, hence $k_\la\in H^2$, and $k_\la$ has the reproducing kernel property for $H^2$:
\[
\tu{span}\{k_\la \colon \la\in\BD\}\mbox{ dense in $H^2$}
\ands
\inn{h}{k_\la}=h(\la) \quad (h\in H^2,\ \la\in\BD).
\]
See \cite{PR16} for a recent account of the theory of reproducing kernel Hilbert spaces and further references.

\begin{lemma}\label{L:eval_funct}
Let $\om=s/q\in\Rat(\BT)$, with $s,q\in\cP$ co-prime, be proper. Then
\[
T_\om k_\la= \overline{\om^*(\la)} k_\la \quad (\la\in\BD).
\]
\end{lemma}

\begin{proof}[\bf Proof]
Suppose $g = T_\om k_\lambda$ then $s(z)(1-\overline{\lambda}z)^{-1} = q(z)g(z) + r(z)$, where $r\in\cP_{m - 1}$. Here $m=\deg(q)$. Hence $(1-\overline{\la}z) g=(s+(1-\overline{\la}z)r)/q$ is in $\Rat(\BT)$ as well as in $H^2$. This can only occur if $(1-\overline{\la}z) g$ is a polynomial, i.e., $g= k_\la \wtil{r}$ for some $\wtil{r}\in\cP$. Thus $s+(1-\overline{\la}z)r=q \wtil{r}$. Since $\om$ is proper, the degree of the left hand side is at most $m$. But then $\wtil{r}$ is constant, say with value $\wtil{c}$. This shows $T_\om k_\la =\wtil{c} k_\la$.

To determine $\wtil{c}$ we evaluate the identity $s+(1-\overline{\la}z)r=q \wtil{c}$ at $1/\overline{\la}$. This gives $s(1/\overline{\la})= q(1/\overline{\la})\wtil{c}$. Note that
\[
s^\sharp(\la)=\la^n \overline{s(1/\overline{\la})} \ands
q^\sharp(\la)=\la^m \overline{q(1/\overline{\la})},
\]
where $n=\deg(s)$. Hence
\[
s(1/\overline{\la})= \overline{\la}^{-n} \overline{s^\sharp(\la)}
\ands
q(1/\overline{\la})= \overline{\la}^{-m} \overline{q^\sharp(\la)}.
\]
This gives
\[
\wtil{c}
=\frac{\overline{\la}^{-n} \overline{s^\sharp(\la)}}{\overline{\la}^{-m} \overline{q^\sharp(\la)}}
=\overline{\left(\frac{\la ^{m-n}s^\sharp(\la)}{q^\sharp(\la)} \right)}=\overline{\om^*(\la)}. \qedhere
\]
\end{proof}

\begin{proof}[\bf Proof of Theorem \ref{T:core}]
It remains to prove that $H(\overline{\BD})$ is a core for $T_\om$ in case $\om$ is proper. So, assume $\om$ is proper. We need to show that the graph of $T_\om|_{H(\overline{\BD})}$ is dense in the graph of $T_\om$. In other words, let $f,g\in H^2$ with $(f,g)$ perpendicular to $G(T_\om|_{H(\overline{\BD})})$, then we need to show $(f,g)$ is perpendicular to $G(T_\om)$. Since $k_\la\in H(\overline{\BD})$, for $\la\in\BD$, we have
\begin{align*}
  0 & = \inn{(f,g)}{(k_\la, T_\om k_\la)}
  =\inn{f}{k_\la} + \inn{g}{\overline{\om^*(\la)} k_\la}
  =f(\la)+ \om^*(\la) g(\la)\quad (\la\in\BD).
\end{align*}
Hence $\om^* g=-f$. In particular, $\om^* g\in H^2$. Thus $g\in \Dom(T_{\om^*}^\tu{Sa})=\Dom(T_\om^*)$ and $T_{\om}^* g =-f$, by Proposition \ref{P:SaToep}. For any $h\in \Dom(T_\om)$ we have
\begin{align*}
\inn{(f,g)}{(h,T_\om h)} &
=\inn{(-T_{\om}^* g,g)}{(h,T_\om h)}
=-\inn{T_\om^* g}{h} +\inn{g}{T_\om h}=0.
\end{align*}
This proves our claim.
\end{proof}

In Section 8 of \cite{S08}, Sarason introduced the class of closed, densely defined operators $T$ on $H^2$ which satisfy
\begin{itemize}
  \item[(1)] $T_z\, \Dom(T) \subset \Dom(T)$;
  \item[(2)] $T_z^* T T_z =T$;
  \item[(3)] $f\in\Dom(T)$, $f(0)=0$ $\Rightarrow$ $T_z^* f \in \Dom(T)$.
\end{itemize}
This class of operators was further studied by Rosenfeld in \cite{R16}, see also \cite{R13}, in which he referred to such operators as Sarason-Toeplitz operators. The operators $T^\tu{Sa}_\vph$, for $\vph\in N^+$, are Sarason-Toeplitz operators, and the class of operators is closed under taking adjoints, by Proposition 2.1 in \cite{R16}. Hence, by Proposition \ref{P:SaToep}, $T_\om$ is a Sarason-Toeplitz operator whenever $\om\in\Rat(\BT)$ is proper. We show that in fact $T_\om$ is a Sarason-Toeplitz operator for any $\om\in\Rat$.

\begin{proposition}
Let $\om\in \Rat$. Then $T_\om$ on $H^2$ is a Sarason-Toeplitz operator.
\end{proposition}

\begin{proof}[\bf Proof]
First consider $\om\in\Rat(\BT)$. That $T_\om$ satisfies (1) and (2) was proved in \cite[Lemma 2.3]{GtHJR1}. We claim that $T_{z}^*\, \Dom(T_\om)\subset \Dom(T_\om)$. Write $\om =s/q$ with $s,q\in \cP$ co-prime. Then $\Dom(T_\om)=q H^2 + \cP_{\deg(q)-1}$. Let $f= q h +r\in\Dom(T_\om)$ with $h\in H^2$ and $r\in\cP$, $\deg(r)<\deg(q)$. Then $T_z^* f= q T_z^* h + h(0) T_z^* q + T_z^* r$, which is in $q H^2+\cP_{\deg(q)-1}=\Dom(T_\om)$. Hence $T_\om$ is a Sarason-Toeplitz operator in case $\om\in \Rat(\BT)$.

Now take $\om\in \Rat$ arbitrarily. By Lemma 5.1 in \cite{GtHJR1}, see also Section \ref{S:general} above, $\om= \om_- z^\kappa \om_0 \om_+$ with $\kappa\in\BZ$, and $\om_-$, $\om_0$ and $\om_+$ in $\Rat$ with zeroes and poles only inside, on or outside $\BT$, respectively. In particular, $\om_0\in\Rat(\BT)$, $\om_-$ and $\om_-^{-1}$ are both anti-analytic, and $\om_+$ and $\om_+^{-1}$ are both analytic. Also, $T_\om=T_{\om_-} T_{z^\kappa \om_0} T_{\om_+}$. Note that $z^\kappa \om_0\in \Rat(\BT)$ in case $\kappa\geq 0$ and  $T_{z^\kappa \om_0}=T_{z^\kappa}T_{\om_0}$ in case $\kappa<0$ (by \cite[Lemma 5.3]{GtHJR1}). In both cases it now easily follows that $T_{z^\kappa\om_0}$ is a Sarason-Toeplitz operator. The claim for $T_\om$ follows since $T_{\om_+}^{\pm 1} T_z=T_z T_{\om_+}^{\pm 1}$ and $T_{\om_-}^{\pm 1} T_z^*=T_z^* T_{\om_-}^{\pm 1}$.
\end{proof}

In fact, by the same arguments one can show that $T_\om$ on $H^p$, $1<p<\infty$, satisfied (1)-(3) in case $T_z^*$ is replaced by $T_{z^{-1}}$.

\end{document}